\documentclass{article}
\ifdefined\pdfoutput\pdfoutput=1\fi
\usepackage{arxiv}
\usepackage{iftex}
\ifPDFTeX
  \usepackage[utf8]{inputenc}
  \usepackage[T1]{fontenc}
\else
  \usepackage{fontspec}
\fi
\usepackage{amsmath,amssymb,amsthm}
\usepackage{stmaryrd}
\usepackage{mathtools}
\usepackage{enumitem}
\usepackage{booktabs}
\usepackage{hyperref}
\hypersetup{
	hidelinks,
	linktoc = all,
	pdfdisplaydoctitle,
	breaklinks,
	pdfstartview = Fit,
	unicode,
	pdftitle={A categorical framework for typed extensional and intensional formal semantics},
	pdfsubject={formal semantics, intensional semantics, model theory},
	pdfauthor={Daniel Quigley},
	pdfcreator = {LuaLaTeX},
	pdfkeywords={formal semantics, intensional model, bounded morphism, logical relation, invariance, mathematical linguistics},
}
\usepackage[activate={true,nocompatibility},final]{microtype}
\usepackage{doi}
\usepackage{array}
\usepackage{cleveref}
\usepackage[most]{tcolorbox}
\usepackage{tikz}
\usetikzlibrary{cd,arrows.meta,positioning,calc}
\DeclareMathOperator{\IM}{im}

\newcommand{\den}[1]{\llbracket #1 \rrbracket}
\newcommand{\Mod}{\mathbf{Mod}}
\newcommand{\Krb}{\mathbf{Kr}_{\mathbf{b}}}
\newcommand{\Fr}{\mathbf{Fr}}
\newcommand{\rel}{\sim}
\newcommand{\Tp}{\mathcal{T}^{+}}
\newcommand{\Tm}{\mathcal{T}^{-}}
\newcommand{\Tc}{\mathcal{T}^{c}}
\newcommand{\Text}{\mathcal{T}_{ext}}
\newcommand{\nec}[1][]{\square_{#1}}
\newcommand{\pos}[1][]{\lozenge_{#1}}
\newcommand{\B}{\mathbf{B}}

\newcommand{\M}{\mathbf{M}}
\newcommand{\Lex}{\mathcal{L}}

\definecolor{boxcolor1}{RGB}{255,254,230}
\definecolor{boxcolor2}{RGB}{252,227,215}
\definecolor{boxcolor3}{RGB}{225,240,227}
\definecolor{boxcolor4}{RGB}{227,235,246}

\tcbset{
	envbox/.style={
		breakable,
		before skip=1.5\topskip,
		after skip=1.5\topskip,
		left skip=0pt,
		right skip=0pt,
		left=4pt,
		right=4pt,
		top=2pt,
		bottom=2pt,
		lefttitle=4pt,
		righttitle=4pt,
		toptitle=2pt,
		bottomtitle=2pt,
		sharp corners,
		boxrule=0pt,
		titlerule=.4pt,
		coltitle=black,
		colframe=darkgray,
		coltext=black,
		fonttitle=\bfseries,
	}
}
\newtcolorbox{definitionbox}[1][]{envbox, colback=boxcolor1, colbacktitle=boxcolor1, #1}
\newtcolorbox{theorembox}[1][]{envbox, colback=boxcolor2, colbacktitle=boxcolor2, #1}
\newtcolorbox{lemmabox}[1][]{envbox, colback=boxcolor3, colbacktitle=boxcolor3, #1}
\newtcolorbox{propositionbox}[1][]{envbox, colback=boxcolor3, colbacktitle=boxcolor3, #1}
\newtcolorbox{corollarybox}[1][]{envbox, colback=boxcolor4, colbacktitle=boxcolor4, #1}

\newtheoremstyle{mystyle}{3pt}{3pt}{\itshape}{}{\bfseries}{.}{.5em}{}
\theoremstyle{mystyle}
\newtheorem{example}{Example}[section]
\newtheorem{remark}{Remark}[section]

\newtheorem{definition}{Definition}[section]
\renewenvironment{definition}[1][]{%
	\refstepcounter{definition}%
	\begin{definitionbox}[title=Definition~\thedefinition\ifx\relax#1\relax\else\ (#1)\fi]%
	}{%
	\end{definitionbox}
}
\newtheorem{theorem}{Theorem}[section]
\renewenvironment{theorem}[1][]{%
	\refstepcounter{theorem}%
	\begin{theorembox}[title=Theorem~\thetheorem\ifx\relax#1\relax\else\ (#1)\fi]%
	}{%
	\end{theorembox}
}
\newtheorem{lemma}{Lemma}[section]
\renewenvironment{lemma}[1][]{%
	\refstepcounter{lemma}%
	\begin{lemmabox}[title=Lemma~\thelemma\ifx\relax#1\relax\else\ (#1)\fi]%
	}{%
	\end{lemmabox}
}
\newtheorem{proposition}{Proposition}[section]
\renewenvironment{proposition}[1][]{%
	\refstepcounter{proposition}%
	\begin{propositionbox}[title=Proposition~\theproposition\ifx\relax#1\relax\else\ (#1)\fi]%
	}{%
	\end{propositionbox}
}
\newtheorem{corollary}{Corollary}[section]
\renewenvironment{corollary}[1][]{%
	\refstepcounter{corollary}%
	\begin{corollarybox}[title=Corollary~\thecorollary\ifx\relax#1\relax\else\ (#1)\fi]%
	}{%
	\end{corollarybox}
}

\title{A categorical framework for typed extensional and intensional formal semantics}

\author{ \href{https://orcid.org/0009-0004-7957-1806}{\includegraphics[scale=0.06]{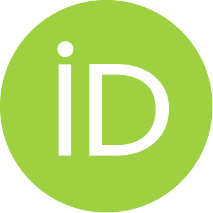}\hspace{1mm}Daniel Quigley} \\
	Center for Possible Minds\\
	Indiana University Bloomington\\
	Bloomington, IN 47408 \\
	\texttt{dgquigle@iu.edu} \\
}

\renewcommand{\headeright}{Preprint}
\renewcommand{\undertitle}{Preprint}
\renewcommand{\shorttitle}{Morphisms of intensional models}

\begin{document}
\maketitle

\begin{abstract}
    A morphism of semantic models is a map under which denotation is invariant. We characterize such morphisms for typed extensional and intensional semantics with finitely many index sorts. Maps on entities and frames induce logical relations on all types, and the fundamental lemma yields invariance when lexical constants are related and frame maps satisfy the appropriate boundedness conditions. The lexicon determines the resulting class of morphisms. Among entity maps, universal and existential determiners force surjectivity, proportional determiners additionally constrain fiber sizes, counting determiners force bijectivity, and identity forces injectivity. Modal operators impose parallel conditions on frame maps: necessity requires bounded morphisms, counting modalities require local bijectivity, and proportional modalities require preservation and reflection of majorities over accessible sets. For lexicons containing the relevant counting determiners, morphisms are bijective on entities. The extensional models then form a groupoid. Intensional morphisms may still act nontrivially on frames. Induced logical relations can fail to compose. We, therefore, identify a composability condition on lexical types under which the models form a category, and polarized types and the standard lexicons considered here satisfy it. We identify ordinary extensional models with the zero-index case for extensional lexicons, determine when index sorts collapse, analyze the effect of enriching the lexicon with index quantification and accessibility, and show that product frames arise from a faithful functor on the corresponding frame categories.
\end{abstract}


\section{Introduction}\label{sec:intro}

Formal semantics interprets a fragment of natural language in a typed model. Extensional models contain entities, truth values, and functions between them \cite{Heim1998}; intensional models add index sorts such as worlds and times, together with accessibility relations \cite{Montague1974,vonFintel2023}. The usual theory studies such models individually. By contrast, we ask which maps between models preserve semantic interpretation.

For frames, part of that question is already settled: a map between Kripke frames preserves modal truth when it is a bounded morphism \cite{blackburn2001modal}. On entities, Tarski's invariance criterion characterizes logical constants by preservation under permutations \cite{Tarski1986}, and van Benthem extends it through the type hierarchy by relating functions that map related arguments to related values \cite{vanBenthem1989}. A map of the entity domain, together with a map of each frame, induces a logical relation at every type. A morphism of models is such a map under which each lexical item is related to its counterpart, and each frame map satisfies the required boundedness condition. The fundamental lemma of logical relations then implies that every expression denotes related values. At type $t$, this relation is equality.

The invariance theorem itself is an instance of the fundamental lemma for the simply typed $\lambda$-calculus \cite{Statman1985}. The question is, therefore, which maps satisfy the lexical conditions; these conditions are determined item by item. For entity-denoting structure, \textit{every} and \textit{some} force surjectivity, \textit{most} forces uniform fiber size in addition to surjectivity, \textit{at least two} and \textit{exactly one} force bijectivity, and identity forces injectivity. The modal cases are parallel: necessity forces bounded morphisms, counting modalities force local bijectivity, and a proportional modality requires preservation and reflection of majorities within each accessible set. The last condition is weaker than uniform fibers because the frame fixes the restrictor of a modal operator. A determiner, instead, takes its restrictor as an argument.

For lexicons containing counting determiners of the relevant kind, these constraints make morphisms bijective on entities; hence, the extensional models form a groupoid. Intensional morphisms over such lexicons may still act nontrivially on frames. The usual extensional models arise as the zero-index case when the lexicon has extensional types. Index sorts can be collapsed only under specific semantic and frame-theoretic conditions, and enriching the lexicon with index quantification and an accessibility constant strengthens the admissible frame maps from bounded morphisms to surjective strong homomorphisms. The strengthening corresponds to the additional expressive resources of Gallin's two-sorted type theory over intensional logic \cite{Gallin2016}; Zimmermann confines that to a decidable class of formulas \cite{Zimmermann1989}.

Categorical closure raises a further issue: induced logical relations fail to compose at some types. At such types, a composite of morphisms can violate lexical relatedness. We, therefore, isolate a class of composable types, and show that the standard lexicons considered here lie within it. The class includes the polarized types, at which the induced relations are graphs of functions. A counterexample at the type of predicates of functions on indices shows that some restriction on lexical types is necessary.

\section{Prior work}\label{sec:priorwork}

Bounded morphisms (p-morphisms) are maps of Kripke frames along which modal truth is preserved and reflected. Their invariance theorem is the bounded morphism lemma \cite{blackburn2001modal}, itself a special case of bisimulation invariance. Thomason showed that the category of frames and bounded morphisms is dually equivalent to a category of complete atomic Boolean algebras with operators \cite{Thomason1975}; Rutten's coalgebraic treatment identifies the same frame category with coalgebras for the covariant powerset functor, with bounded morphisms as coalgebra homomorphisms \cite{Rutten2000}. These results concern propositional modal logic over a single frame. Intensional semantic models, instead, combine several frames with a typed entity domain and higher-type lexical denotations, within which relatedness of the necessity operator is equivalent to the usual bounded morphism condition.

A logical relation lifts a relation on the base types through function types by relating functions that map related arguments to related values; Plotkin uses such relations to characterize $\lambda$-definability \cite{Plotkin1980}. Statman's fundamental lemma states that the lifted relation contains the denotations of every closed term of the simply typed $\lambda$-calculus \cite{Statman1985}. We follow Mitchell's presentation \cite{Mitchell1996}: the invariance theorem for morphisms of models is this lemma applied to intensional semantics, with modal operators treated as constants\footnote{By itself, the lemma leaves categorical closure unsettled, because composites of logical relations need not again be logical relations.}. 

Tarski proposed invariance under permutations of a domain as a criterion for logicality \cite{Tarski1986}, and Sher extended it to bijections between domains \cite{Sher1991}; Van Benthem lifts permutations through the type hierarchy \cite{vanBenthem1989}. Applied to bijections between domains, the same lifting gives the bijective case of the relations induced here. Keenan and Stavi apply permutation invariance to determiners, including \textit{every}, \textit{some}, \textit{most}, and counting determiners \cite{KeenanStavi1986}; Feferman objects, that the resulting criterion assimilates logic to mathematics, and fails to say what makes an operation the same across domains of different sizes \cite{Feferman1999}; his alternative, invariance under homomorphisms, excludes identity and the counting quantifiers. Casanovas studies operators on a finite relational type hierarchy that are invariant under arbitrary surjections, and characterizes them by definability in a fragment of first-order logic \cite{Casanovas2007}; Bonnay argues for invariance under potential isomorphism \cite{Bonnay2008}. A potential isomorphism is a family of partial isomorphisms with the back-and-forth property. The functional maps here exclude such families; instead, different fragments of the lexicon determine different admissible entity maps: on index sorts, necessity selects bounded morphisms, and counting modalities from graded modal logic \cite{deRijke2000} impose on frame maps the analogue of the constraints that counting determiners place on entity maps.

Gallin embeds Montague's intensional logic into the two-sorted type theory $\mathrm{Ty}_2$, where indices form a primitive type and necessity is expressed by quantification over them \cite{Gallin2016}; Zimmermann shows that, apart from a decidable class of formulas without linguistic application, $\mathrm{Ty}_2$ translates back into intensional logic with validity and entailment preserved \cite{Zimmermann1989}. A model-theoretic counterpart holds over fixed underlying models: the richer two-sorted lexicon admits fewer morphisms, and the separating formula concerns the index structure itself. Product frames and their commutation axioms are studied in \cite{gabbay2003many}; we use the product construction only to the extent that the several-sorted models of \cite{quigley2026intensional} require.

\section{Models}\label{sec:models}

The definitions follow \cite{quigley2026intensional}, with the measure structure omitted, the set of sorts allowed to be empty, and variables admitted at every type. We work in ordinary set theory with the axiom of choice.

\begin{definition}[Sorts and types]\label{def:types}
    Let $\Sigma$ be a finite set of \textbf{index sorts} (possibly empty). The set $\mathcal{T}$ of \textbf{types} is the smallest set containing $e$, $t$, each $\sigma \in \Sigma$, and $s$, and closed under $\langle \alpha, \beta \rangle$ for $\alpha, \beta \in \mathcal{T}$. The \textbf{extensional types} $\Text$ are those in which neither $s$ nor any $\sigma$ occurs. The \textbf{base types} are $e$, $t$, the sorts, and $s$.
\end{definition}

\begin{definition}[Frame]\label{def:frame}
    A \textbf{frame} is a pair $\mathcal{F} = (\mathcal{D}, \mathcal{R})$, where $\mathcal{D}$ is a nonempty set and $\mathcal{R}$ is a binary relation on $\mathcal{D}$. Write
    \[
        \mathcal{R}(w) := \{w' : \mathcal{R}(w,w')\}.
    \]
    
\end{definition}

\begin{definition}[Bounded morphisms]\label{def:bounded}
    A function $f : \mathcal{D} \to \mathcal{D}'$ is a \textbf{bounded morphism}
    \[
        (\mathcal{D},\mathcal{R}) \to (\mathcal{D}',\mathcal{R}')
    \]
    if it satisfies the \textbf{forth} condition, that $\mathcal{R}(w,v)$ implies $\mathcal{R}'(f(w),f(v))$, and the \textbf{back} condition, that if $\mathcal{R}'(f(w),v')$, then $\mathcal{R}(w,v)$ and $f(v)=v'$ for some $v$.
    Equivalently,
    \[
        f[\mathcal{R}(w)] = \mathcal{R}'(f(w))
    \]
    for every $w$, or, reading $f$ as a relation,
    \[
        f \circ \mathcal{R} = \mathcal{R}' \circ f.
    \]
\end{definition}

Frames and bounded morphisms form the category $\Krb$. The forth condition alone is $f \circ \mathcal{R} \subseteq \mathcal{R}' \circ f$; functions satisfying it form the larger category $\mathbf{Kr}$.

\begin{definition}[Lexicon]\label{def:lexicon}
    A \textbf{lexicon} $\Lex$ is a set of typed constants; $\Lex_\tau$ is the set of constants of type $\tau$. 
\end{definition}   
    
\begin{definition}[Lexicon and model]\label{def:model}
    An \textbf{intensional model} over $\Sigma$ and $\Lex$ is a tuple
    \[
        \mathcal{M} =
        \langle
            (\mathcal{D}_\tau)_{\tau \in \mathcal{T}},
            (\mathcal{F}_\sigma)_{\sigma \in \Sigma},
            S,
            \mathcal{I}
        \rangle
    \]
    such that:
    \begin{enumerate}[label=(\roman*)]
        \item $\mathcal{D}_e$ is a nonempty set of entities and $\mathcal{D}_t=\{0,1\}$;
        \item for each $\sigma \in \Sigma$, $\mathcal{F}_\sigma=(\mathcal{D}_\sigma,\mathcal{R}_\sigma)$ is a frame;
        \item $S=\prod_{\sigma\in\Sigma}\mathcal{D}_\sigma$ is the set of \textbf{compound indices}, and $\mathcal{D}_s=S$;
        \item $\mathcal{D}_{\langle\alpha,\beta\rangle}$ is the set of all functions $\mathcal{D}_\alpha\to\mathcal{D}_\beta$;
        \item $\mathcal{I}$ assigns each $c\in\Lex_\tau$ an \textbf{intension}
        \[
            \mathcal{I}(c):S\to\mathcal{D}_\tau.
        \]
    \end{enumerate}
    For $s\in S$, $\sigma\in\Sigma$, and $w\in\mathcal{D}_\sigma$, let $s[\sigma\mapsto w]$ agree with $s$ except at $\sigma$, where it takes value $w$.
\end{definition}   

If $\Sigma=\emptyset$, then $S$ is the singleton containing the empty tuple. Each intension accordingly identifies canonically with an element of its value domain. We call the resulting structure a \textbf{zero-index model}. If $\Lex$ has only extensional types, it is the usual extensional model
\[
    \mathcal{M} =\langle  (\mathcal{D}_\tau)_{\tau\in\Text},\mathcal{I}\rangle.
\]

\begin{definition}[Expressions and denotation]\label{def:denotation}
    For each type $\tau$, let there be countably many variables of type $\tau$. The \textbf{expressions} over $\Lex$ are generated by constants and variables, application, abstraction, and, for $\sigma\in\Sigma$, formulas of the form $\nec[\sigma]\varphi$.

    An \textbf{assignment} $g$ sends each variable of type $\tau$ to an element of $\mathcal{D}_\tau$; write $g[x\mapsto d]$ for its variant at $x$. The \textbf{denotation} $\den{\xi}^{\mathcal{M},g,s}$ of an expression $\xi$ at $s\in S$ is defined recursively:
    \begin{enumerate}[label=(\arabic*)]
        \item $\den{c}^{\mathcal{M},g,s}=\mathcal{I}(c)(s)$;
        \item $\den{x}^{\mathcal{M},g,s}=g(x)$;
        \item $\den{\alpha(\beta)}^{\mathcal{M},g,s}
        =
        \den{\alpha}^{\mathcal{M},g,s}
        \bigl(\den{\beta}^{\mathcal{M},g,s}\bigr)$;
        \item $\den{\lambda x.\psi}^{\mathcal{M},g,s}$ is the function
        \[
            d\mapsto \den{\psi}^{\mathcal{M},g[x\mapsto d],s};
        \]
        \item
        \[
            \den{\nec[\sigma]\varphi}^{\mathcal{M},g,s}=1
        \]
        iff
        \[
            \den{\varphi}^{\mathcal{M},g,s[\sigma\mapsto w']}=1
        \]
        for every $w'\in\mathcal{R}_\sigma(s_\sigma)$.
    \end{enumerate}
    Define $\pos[\sigma]\varphi:=\neg\nec[\sigma]\neg\varphi$. A \textbf{formula} is an expression of type $t$; $\mathcal{M},s\models\varphi$ means $\den{\varphi}^{\mathcal{M},g,s}=1$ for every assignment $g$, and $\varphi$ is \textbf{valid} in $\mathcal{M}$ if $\mathcal{M},s\models\varphi$ for every $s$.
\end{definition}

A constant is standardly interpreted when its intension is constant in $s$ and its value is the function conventionally designated by that constant on the model's domains. Standard interpretation is imposed only where explicitly assumed. The semantic clause for necessity is retained in primitive form here; at each index, it agrees with an operator of type $\langle\langle s,t\rangle,t\rangle$ applied to the intension of the embedded formula.

\section{Induced relations}\label{sec:relations}

A map between models is specified on the base types and extended through the type hierarchy by the logical-relations clause of \cite{vanBenthem1989}.

\begin{definition}[Sort map]\label{def:sortmap}
    Let $\mathcal{M}$ and $\mathcal{M}'$ be models over the same $\Sigma$ and $\Lex$. A \textbf{sort map} $f : \mathcal{M} \to \mathcal{M}'$ is a family of functions
    \[
        f_e : \mathcal{D}_e \to \mathcal{D}'_e,
        \qquad
        f_\sigma : \mathcal{D}_\sigma \to \mathcal{D}'_\sigma
        \quad (\sigma \in \Sigma).
    \]
    It determines
    \[
        f_s := \prod_{\sigma} f_\sigma : S \to S',
        \qquad
        f_t := \mathrm{id}_{\{0,1\}},
    \]
    with
    \[
        f_s(s[\sigma \mapsto w])
        =
        f_s(s)[\sigma \mapsto f_\sigma(w)].
    \]
    Sort maps compose componentwise.
\end{definition}

\begin{definition}[Induced relation]\label{def:relation}
    A sort map $f$ induces, for each type $\tau$, a relation
    \[
        \rel^{f}_\tau
        \subseteq
        \mathcal{D}_\tau \times \mathcal{D}'_\tau.
    \]
    At a base type $\tau$,
    \[
        a \rel^{f}_\tau a'
        \quad\text{iff}\quad
        f_\tau(a)=a'.
    \]
    At a function type,
    \[
        F \rel^{f}_{\langle\alpha,\beta\rangle} F'
    \]
    iff
    \[
        F(a) \rel^{f}_\beta F'(a')
    \]
    for every $a \rel^{f}_\alpha a'$.
\end{definition}

We omit the superscript when $f$ is clear. At type $t$, the induced relation is equality; its direction at higher types depends on the variance of the base sorts occurring in them. For example, if $\tau$ is extensional and $f_e$ is the identity, then at type $\langle\sigma,\tau\rangle$,
\[
    F \rel F'
    \quad\text{iff}\quad
    F=F'\circ f_\sigma.
\]

\begin{definition}[Polarized types]\label{def:polarized}
    The classes $\Tp$ and $\Tm$ of \textbf{covariant} and \textbf{contravariant} types are the smallest classes such that:
    \begin{enumerate}[label=(\roman*)]
        \item $t \in \Tp \cap \Tm$;
        \item $e$, each $\sigma \in \Sigma$, and $s$ lie in $\Tp$;
        \item if $\alpha \in \Tm$ and $\beta \in \Tp$, then $\langle\alpha,\beta\rangle \in \Tp$;
        \item if $\alpha \in \Tp$ and $\beta \in \Tm$, then $\langle\alpha,\beta\rangle \in \Tm$.
    \end{enumerate}
    A type is \textbf{polarized} if it lies in $\Tp \cup \Tm$. The class $\Tc$ of \textbf{composable} types is the smallest class containing $\Tp \cup \Tm$ and satisfying
    \[
        \alpha \in \Tp \cup \Tm,\ \beta \in \Tc
        \quad\Longrightarrow\quad
        \langle\alpha,\beta\rangle \in \Tc.
    \]
\end{definition}

Equivalently, a type is covariant when every occurrence of a base sort other than $t$ lies under an even number of argument positions, and contravariant when every such occurrence lies under an odd number. Thus, propositions $\langle s,t\rangle$ and predicates $\langle e,t\rangle$ are contravariant, and generalized quantifiers $\langle\langle e,t\rangle,t\rangle$ are covariant. The type $\langle\sigma,\sigma\rangle$ is neither, and $\langle\langle\sigma,\sigma\rangle,t\rangle$ is not composable. The intension type $\langle s,\tau\rangle$ is composable whenever $\tau$ is.

\begin{lemma}[Polarity]\label{lem:polarity}
    Let $f : \mathcal{M} \to \mathcal{M}'$ and $f' : \mathcal{M}' \to \mathcal{M}''$ be sort maps.
    \begin{enumerate}[label=(\arabic*)]
        \item If $\tau \in \Tp$, there is a function
        \[
            f^{+}_\tau : \mathcal{D}_\tau \to \mathcal{D}'_\tau
        \]
        whose graph is $\rel^{f}_\tau$. If $\tau \in \Tm$, there is a function
        \[
            f^{-}_\tau : \mathcal{D}'_\tau \to \mathcal{D}_\tau
        \]
        whose graph is the converse of $\rel^{f}_\tau$. On base types $f^{+}_\tau=f_\tau$, and on $t$ both functions are the identity. Moreover,
        \[
            f^{+}_{\langle\alpha,\beta\rangle}(F)
            =
            f^{+}_\beta \circ F \circ f^{-}_\alpha,
            \qquad
            f^{-}_{\langle\alpha,\beta\rangle}(F')
            =
            f^{-}_\beta \circ F' \circ f^{+}_\alpha.
        \]
        \item
        \[
            (f'\circ f)^{+}_\tau
            =
            f'^{+}_\tau \circ f^{+}_\tau
            \qquad (\tau\in\Tp),
        \]
        and
        \[
            (f'\circ f)^{-}_\tau
            =
            f^{-}_\tau \circ f'^{-}_\tau
            \qquad (\tau\in\Tm).
        \]
        \item If $\tau\in\Tp\cup\Tm$, then
        \[
            \rel^{f'}_\tau \circ \rel^{f}_\tau
            =
            \rel^{f'\circ f}_\tau.
        \]
        If $\tau\in\Tc$, then
        \[
            \rel^{f'}_\tau \circ \rel^{f}_\tau
            \subseteq
            \rel^{f'\circ f}_\tau.
        \]
        For every type $\tau$, $\rel^{\mathrm{id}}_\tau$ is the identity relation.
    \end{enumerate}
\end{lemma}

\begin{proof}
    (1) By induction on the generation of $\Tp$ and $\Tm$. The base cases are immediate. Suppose $\langle\alpha,\beta\rangle\in\Tp$, so $\alpha\in\Tm$ and $\beta\in\Tp$. The pairs $a\rel_\alpha a'$ are exactly $(f^{-}_\alpha(a'),a')$. Hence
    \[
        F\rel F'
    \]
    iff
    \[
        F(f^{-}_\alpha(a'))\rel_\beta F'(a')
    \]
    for every $a'$, iff
    \[
        F'
        =
        f^{+}_\beta\circ F\circ f^{-}_\alpha.
    \]
    Thus each $F$ is related to exactly one $F'$. The contravariant case is symmetric and gives
    \[
        F
        =
        f^{-}_\beta\circ F'\circ f^{+}_\alpha.
    \]

    (2) Again by induction. For a covariant function type,
    \[
    \begin{aligned}
        (f'\circ f)^{+}_{\langle\alpha,\beta\rangle}(F)
        &=
        (f'\circ f)^{+}_\beta
        \circ F
        \circ
        (f'\circ f)^{-}_\alpha \\
        &=
        f'^{+}_\beta
        \circ f^{+}_\beta
        \circ F
        \circ f^{-}_\alpha
        \circ f'^{-}_\alpha \\
        &=
        f'^{+}_{\langle\alpha,\beta\rangle}
        \bigl(
            f^{+}_{\langle\alpha,\beta\rangle}(F)
        \bigr).
    \end{aligned}
    \]
    The contravariant case is analogous; on base types the claim is componentwise composition.

    (3) For $\tau\in\Tp$, the composite of the graphs of $f^{+}_\tau$ and $f'^{+}_\tau$ is the graph of
    \[
        f'^{+}_\tau\circ f^{+}_\tau
        =
        (f'\circ f)^{+}_\tau,
    \]
    and similarly for $\tau\in\Tm$ using converse graphs.

    For $\tau\in\Tc$, proceed by induction. The polarized case is the equality proved above. Suppose $\tau=\langle\alpha,\beta\rangle$ with $\alpha\in\Tp\cup\Tm$ and $\beta\in\Tc$, and let
    \[
        F\rel^{f}F'\rel^{f'}F'',
        \qquad
        a\rel^{f'\circ f}_\alpha a''.
    \]
    By equality at $\alpha$, there is $a'$ with
    \[
        a\rel^{f}_\alpha a'\rel^{f'}_\alpha a''.
    \]
    Hence
    \[
        F(a)\rel^{f}_\beta F'(a')
        \rel^{f'}_\beta F''(a''),
    \]
    and the induction hypothesis gives
    \[
        F(a)\rel^{f'\circ f}_\beta F''(a'').
    \]
    Thus $F\rel^{f'\circ f}F''$.

    Finally, $\rel^{\mathrm{id}}_\tau$ is the identity relation by induction on $\tau$.
\end{proof}

The inclusion in (3) of Lemma~\ref{lem:polarity} is sufficient for composition of morphisms, and can fail outside $\Tc$.

\begin{lemma}[Surjective sort maps]\label{lem:surjective}
    If every component of $f$ is surjective, then $f^{+}_\tau$ is surjective for every $\tau\in\Tp$, and $f^{-}_\tau$ is injective for every $\tau\in\Tm$.
\end{lemma}

\begin{proof}
    By induction on the generation of $\Tp$ and $\Tm$. On base types the claim follows from the hypothesis, and on $t$ both functions are the identity.

    Let $\langle\alpha,\beta\rangle\in\Tp$. Then $f^{-}_\alpha$ is injective and $f^{+}_\beta$ is surjective. Given $F' \in \mathcal{D}'_{\langle \alpha, \beta \rangle}$, choose, using the standing choice assumption, a left inverse $\ell$ of $f^{-}_\alpha$ and a right inverse $r$ of $f^{+}_\beta$, and set
    \[
        F:=r\circ F'\circ\ell.
    \]
    Then
    \[
        f^{+}_\beta\circ F\circ f^{-}_\alpha=F',
    \]
    so $f^{+}_{\langle\alpha,\beta\rangle}$ is surjective.

    Let $\langle\alpha,\beta\rangle\in\Tm$. Then $f^{+}_\alpha$ is surjective and $f^{-}_\beta$ is injective. If
    \[
        f^{-}_\beta\circ F'_1\circ f^{+}_\alpha
        =
        f^{-}_\beta\circ F'_2\circ f^{+}_\alpha,
    \]
    injectivity of $f^{-}_\beta$ and surjectivity of $f^{+}_\alpha$ yield $F'_1=F'_2$.
\end{proof}

\begin{example}[Failure of composition]\label{ex:failure}
    Let $\Sigma=\{\sigma\}$, and let $\mathcal{M}$, $\mathcal{M}'$, and $\mathcal{M}''$ have frames with domains
    \[
        \mathcal{D}_\sigma=\{a_1,a_2,b_1,b_2\},
        \qquad
        \mathcal{D}'_\sigma=\{a,b\},
        \qquad
        \mathcal{D}''_\sigma=\{\ast\},
    \]
    each equipped with the total relation, and with a common entity domain. Let
    \[
        f_\sigma(a_i)=a,\qquad
        f_\sigma(b_i)=b,
    \]
    and let $f'_\sigma$ send both $a$ and $b$ to $\ast$. Both maps are surjective bounded morphisms; take $f_e=f'_e=\mathrm{id}$.

    Let $\Lex$ contain a constant $G$ of type
    \[
        \langle\langle\sigma,\sigma\rangle,t\rangle
    \]
    with intensions constant in $s$ and defined by
    \[
        \mathcal{I}(G)(s)(F)=1
        \quad\text{iff}\quad
        f_\sigma(F(a_1))\neq f_\sigma(F(a_2)),
    \]
    while
    \[
        \mathcal{I}'(G)(s')(F')=0,
        \qquad
        \mathcal{I}''(G)(s'')(F'')=0.
    \]

    At type $\langle\sigma,\sigma\rangle$,
    \[
        F\rel^{f}F'
        \quad\text{iff}\quad
        f_\sigma\circ F=F'\circ f_\sigma.
    \]
    This implies
    \[
        f_\sigma(F(a_1))=f_\sigma(F(a_2)),
    \]
    so $\mathcal{I}(G)(s)(F)=0=\mathcal{I}'(G)(f_s(s))(F')$ for every related pair. Hence
    \[
        \mathcal{I}(G)\rel^{f}\mathcal{I}'(G).
    \]
    Every $F'$ is related under $f'$ to the unique function on $\{\ast\}$, and both $\mathcal{I}'(G)$ and $\mathcal{I}''(G)$ are constantly $0$, so
    \[
        \mathcal{I}'(G)\rel^{f'}\mathcal{I}''(G).
    \]

    Under $f'\circ f$, however, every $F$ is related to the unique function on $\{\ast\}$. Thus
    \[
        \mathcal{I}(G)\rel^{f'\circ f}\mathcal{I}''(G)
    \]
    would require $\mathcal{I}(G)(s)(F)=0$ for every $F$. But if
    \[
        F(a_1)=a_1,
        \qquad
        F(a_2)=b_1,
    \]
    then $\mathcal{I}(G)(s)(F)=1$. The interpretations of $G$ are, thus, related under $f$ and under $f'$. Under $f'\circ f$, they are unrelated.
\end{example}

In Example~\ref{ex:failure}, the argument type $\langle\sigma,\sigma\rangle$ of $G$ is neither covariant nor contravariant. Standard nouns, verbs, determiners, attitude verbs, and intensional transitive verbs typed as relations to properties \cite{Zimmermann1993}, as well as quantifiers and accessibility constants over index sorts, have composable types. Thus, the standard lexicons considered here lie within $\Tc$. Montague types intensional transitive verbs as functions on the intension of a quantifier over properties \cite{Montague1973}; that type lies outside $\Tc$, because its argument type is not polarized. Indeed, some restriction on lexical types is nevertheless necessary.

\section{Morphisms}\label{sec:morphisms}

A morphism is a sort map that preserves both the lexical interpretation and the modal structure.

\begin{definition}[Morphism]\label{def:morphism}
    Let $\mathcal{M}$ and $\mathcal{M}'$ be models over $\Sigma$ and $\Lex$. A sort map $f : \mathcal{M} \to \mathcal{M}'$ is a \textbf{morphism} if:
    \begin{enumerate}[label=(M\arabic*)]
        \item each $f_\sigma$ is a bounded morphism $\mathcal{F}_\sigma \to \mathcal{F}'_\sigma$;
        \item for every $c \in \Lex_\tau$,
        \[
            \mathcal{I}(c) \rel^{f}_{\langle s,\tau\rangle} \mathcal{I}'(c),
        \]
        equivalently,
        \[
            \mathcal{I}(c)(s)
            \rel^{f}_\tau
            \mathcal{I}'(c)(f_s(s))
        \]
        for every $s\in S$.
    \end{enumerate}
\end{definition}

If $\tau$ is contravariant, (M2) says
\[
    \mathcal{I}(c)(s)
    =
    f^{-}_\tau\bigl(\mathcal{I}'(c)(f_s(s))\bigr);
\]
if $\tau$ is covariant, it says
\[
    f^{+}_\tau\bigl(\mathcal{I}(c)(s)\bigr)
    =
    \mathcal{I}'(c)(f_s(s)).
\]
Thus lexical interpretation is pulled back in the contravariant case and transported forward in the covariant case, as in Figure~\ref{fig:squares}. For a propositional constant $p$ of type $t$, (M2) reduces to
\[
    \mathcal{I}(p)=\mathcal{I}'(p)\circ f_s,
\]
the usual valuation condition for bounded morphisms of Kripke models \cite{blackburn2001modal}.

\begin{figure}[h]
    \centering
    \begin{tikzcd}[column sep=large, row sep=large]
        S \arrow[r, "f_s"] \arrow[d, "\mathcal{I}(c)"'] & S' \arrow[d, "\mathcal{I}'(c)"] \\
        \mathcal{D}_\tau & \mathcal{D}'_\tau \arrow[l, "f^{-}_\tau"]
    \end{tikzcd}
    \qquad\qquad
    \begin{tikzcd}[column sep=large, row sep=large]
        S \arrow[r, "f_s"] \arrow[d, "\mathcal{I}(c)"'] & S' \arrow[d, "\mathcal{I}'(c)"] \\
        \mathcal{D}_\tau \arrow[r, "f^{+}_\tau"'] & \mathcal{D}'_\tau
    \end{tikzcd}
    \caption{Condition (M2) for a constant of contravariant type (left) and covariant type (right).}
    \label{fig:squares}
\end{figure}

\begin{lemma}[Necessity]\label{lem:box}
    In a model $\mathcal{M}$, let
    \[
        \B_\sigma
        \in
        \mathcal{D}_{\langle s,\langle\langle s,t\rangle,t\rangle\rangle}
    \]
    be defined by
    \[
        \B_\sigma(s)(p)=1
    \]
    iff
    \[
        p(s[\sigma\mapsto w'])=1
    \]
    for every $w'\in\mathcal{R}_\sigma(s_\sigma)$. Then
    \[
        \den{\nec[\sigma]\varphi}^{\mathcal{M},g,s}
        =
        \B_\sigma(s)\bigl(\den{\varphi}^{\mathcal{M},g,\cdot}\bigr),
    \]
    and, for a sort map $f:\mathcal{M}\to\mathcal{M}'$,
    \[
        \B_\sigma\rel^{f}\B'_\sigma
    \]
    iff $f_\sigma$ is a bounded morphism.
\end{lemma}

\begin{proof}
    The first claim is clause~(5) of Definition~\ref{def:denotation}. Since $\langle s,t\rangle$ is contravariant,
    \[
        f^{-}_{\langle s,t\rangle}(p')
        =
        p'\circ f_s.
    \]
    Hence $\B_\sigma\rel\B'_\sigma$ iff, for every $s\in S$ and $p':S'\to\{0,1\}$,
    \begin{equation}\label{eq:boxrel}
        \bigl[
        p'(f_s(s)[\sigma\mapsto f_\sigma(w')])=1
        \text{ for all }w'\in\mathcal{R}_\sigma(s_\sigma)
        \bigr]
        =
        \bigl[
        p'(f_s(s)[\sigma\mapsto v'])=1
        \text{ for all }v'\in\mathcal{R}'_\sigma(f_\sigma(s_\sigma))
        \bigr].
    \end{equation}

    If $f_\sigma$ is bounded, the back condition gives the implication from left to right in \eqref{eq:boxrel}, and the forth condition gives the converse.

    Conversely, suppose \eqref{eq:boxrel} holds for every $s$ and $p'$, and fix $w=s_\sigma$. As $p'$ varies, the set
    \[
        A
        :=
        \{v' : p'(f_s(s)[\sigma\mapsto v'])=1\}
    \]
    ranges over all subsets of $\mathcal{D}'_\sigma$. Thus \eqref{eq:boxrel} is equivalent to
    \[
        \bigl[
            f_\sigma[\mathcal{R}_\sigma(w)]
            \subseteq A
        \bigr]
        =
        \bigl[
            \mathcal{R}'_\sigma(f_\sigma(w))
            \subseteq A
        \bigr]
    \]
    for every $A\subseteq\mathcal{D}'_\sigma$. Taking
    \[
        A=\mathcal{R}'_\sigma(f_\sigma(w))
    \]
    gives the forth condition, and taking
    \[
        A=f_\sigma[\mathcal{R}_\sigma(w)]
    \]
    gives the reverse inclusion, hence the back condition.
\end{proof}

Thus (M1) is precisely the relatedness condition for the necessity operators. Equivalently, the morphisms are the sort maps under which every member of
\[
    \Lex\cup\{\B_\sigma:\sigma\in\Sigma\}
\]
is related to its counterpart.

\begin{theorem}[Category]\label{thm:category}
    If every type occurring in $\Lex$ is composable, then models over $\Sigma$ and $\Lex$, with morphisms as in Definition~\ref{def:morphism}, form a category $\Mod_\Sigma(\Lex)$ under composition of sort maps.
\end{theorem}

\begin{proof}
    The identity sort map satisfies (M1), and it satisfies (M2) because $\rel^{\mathrm{id}}$ is the identity relation at every type by (3) of Lemma~\ref{lem:polarity}.

    Let
    \[
        f:\mathcal{M}\to\mathcal{M}',
        \qquad
        f':\mathcal{M}'\to\mathcal{M}''
    \]
    be morphisms. Bounded morphisms compose, so $f'\circ f$ satisfies (M1). For $c\in\Lex_\tau$, the intension type $\langle s,\tau\rangle$ is composable because $s$ is polarized and $\tau$ is composable. Hence
    \[
        \mathcal{I}(c)
        \rel^{f}
        \mathcal{I}'(c)
        \rel^{f'}
        \mathcal{I}''(c)
    \]
    implies
    \[
        \mathcal{I}(c)
        \rel^{f'\circ f}
        \mathcal{I}''(c)
    \]
    by (3) of Lemma~\ref{lem:polarity}. Thus, the composite satisfies (M2). Associativity and the unit laws follow from those of functions.
\end{proof}

Example~\ref{ex:failure} shows that composability cannot be omitted in general: for the lexicon $\{G\}$, morphisms are not closed under composition. Whether some weaker condition on the lexicon also suffices for closure remains for future work.

\section{Invariance}\label{sec:invariance}

Assignments $g$ in $\mathcal{M}$ and $g'$ in $\mathcal{M}'$ are related under a sort map $f$, written $g \rel^{f} g'$, if
\[
    g(x)\rel^{f}_\tau g'(x)
\]
for every variable $x$ of type $\tau$. Every induced relation $\rel^{f}_\tau$ is nonempty, by induction on $\tau$. Related assignments therefore exist.

\begin{theorem}[Invariance]\label{thm:invariance}
    Let $f : \mathcal{M} \to \mathcal{M}'$ be a morphism and $g \rel^{f} g'$. For every expression $\xi$ of type $\tau$ and every $s \in S$,
    \[
        \den{\xi}^{\mathcal{M}, g, s}
        \rel^{f}_\tau
        \den{\xi}^{\mathcal{M}', g', f_s(s)}.
    \]
\end{theorem}

\begin{proof}
    Again, by induction on $\xi$, simultaneously for all $s$ and all related assignments. If $\xi=c\in\Lex_\tau$, the claim is (M2). If $\xi$ is a variable, it follows from $g\rel g'$.

    If $\xi=\alpha(\beta)$, the induction hypothesis gives
    \[
        \den{\alpha}^{\mathcal{M},g,s}
        \rel_{\langle\tau_1,\tau\rangle}
        \den{\alpha}^{\mathcal{M}',g',f_s(s)}
    \]
    and
    \[
        \den{\beta}^{\mathcal{M},g,s}
        \rel_{\tau_1}
        \den{\beta}^{\mathcal{M}',g',f_s(s)}.
    \]
    Definition~\ref{def:relation} then relates their values.

    If $\xi=\lambda x.\psi$, let $d\rel_{\tau_1}d'$. Then
    \[
        g[x\mapsto d]\rel g'[x\mapsto d'],
    \]
    so the induction hypothesis gives
    \[
        \den{\psi}^{\mathcal{M},g[x\mapsto d],s}
        \rel_{\tau_2}
        \den{\psi}^{\mathcal{M}',g'[x\mapsto d'],f_s(s)}.
    \]
    Since $d\rel d'$ was arbitrary, the resulting functions are related.

    If $\xi=\nec[\sigma]\varphi$, the induction hypothesis and equality at type $t$ give
    \[
        \den{\varphi}^{\mathcal{M},g,u}
        =
        \den{\varphi}^{\mathcal{M}',g',f_s(u)}
    \]
    for every $u\in S$. Thus, with
    \[
        p:=\den{\varphi}^{\mathcal{M},g,\cdot},
        \qquad
        p':=\den{\varphi}^{\mathcal{M}',g',\cdot},
    \]
    we have $p=p'\circ f_s$, hence
    \[
        p\rel_{\langle s,t\rangle}p'.
    \]
    By Lemma~\ref{lem:box} and (M1),
    \[
        \B_\sigma\rel\B'_\sigma,
    \]
    so
    \[
        \B_\sigma(s)(p)
        =
        \B'_\sigma(f_s(s))(p'),
    \]
    which is the required equality.
\end{proof}

Apart from the modal clause, this is the fundamental lemma of logical relations \cite{Statman1985}; Lemma~\ref{lem:box} supplies the modal case.

\begin{corollary}[Transfer]\label{cor:transfer}
    Let $f : \mathcal{M} \to \mathcal{M}'$ be a morphism.
    \begin{enumerate}[label=(\arabic*)]
        \item For every closed formula $\varphi$ and $s\in S$,
        \[
            \mathcal{M},s\models\varphi
            \quad\Longleftrightarrow\quad
            \mathcal{M}',f_s(s)\models\varphi.
        \]
        Hence, validity transfers from $\mathcal{M}'$ to $\mathcal{M}$, and conversely when $f_s$ is surjective.
        \item If $\xi$ is a closed expression of contravariant type $\tau$, then
        \[
            \den{\xi}^{\mathcal{M},s}
            =
            f^{-}_\tau
            \bigl(
                \den{\xi}^{\mathcal{M}',f_s(s)}
            \bigr).
        \]
        If $\xi$ has covariant type,
        \[
            f^{+}_\tau
            \bigl(
                \den{\xi}^{\mathcal{M},s}
            \bigr)
            =
            \den{\xi}^{\mathcal{M}',f_s(s)}.
        \]
        \item If $f_e=\mathrm{id}$, then for every closed expression $\xi$ of extensional type,
        \[
            \den{\xi}^{\mathcal{M},\cdot}
            =
            \den{\xi}^{\mathcal{M}',\cdot}\circ f_s.
        \]
    \end{enumerate}
\end{corollary}

\begin{proof}
    For (1), the relation at $t$ is equality, and closed formulas are assignment-independent. Quantifying over $s$ gives the claim about validity; surjectivity of $f_s$ gives the converse direction.

    Part~(2) follows from Theorem~\ref{thm:invariance} and (1) of Lemma~\ref{lem:polarity}. For (3), if $f_e=\mathrm{id}$, then the induced relation at every extensional type is equality.
\end{proof}

\begin{remark}[Forth alone]\label{rem:monotone}
    If the components of a sort map are morphisms of $\mathbf{Kr}$, and constants of type $t$ satisfy (M2), then truth is preserved in one direction for positive modal fragments. For formulas built from propositional constants by conjunction, disjunction, and possibility,
    \[
        \mathcal{M},s\models\varphi
        \quad\Longrightarrow\quad
        \mathcal{M}',f_s(s)\models\varphi.
    \]
    Replacing possibility with necessity reverses the implication. The two directions coincide when the back condition also holds, as Lemma~\ref{lem:box} shows.
\end{remark}

\section{Determination}\label{sec:determination}

For an arbitrarily interpreted lexical item, (M2) relates its interpretations in two models. When an item is standardly interpreted in both models, (M2) instead imposes a condition on the sort map. Let us computes those conditions for entity and index sorts. Throughout, $f:\mathcal{M}\to\mathcal{M}'$ is a sort map, $f^{-1}(A)$ denotes preimage under the relevant component of $f$, and a fiber is the preimage of a point.

On entities, identity has type $\langle e,\langle e,t\rangle\rangle$, and is interpreted by
\[
    \mathcal{I}(=)(s)(a)(b)=[a=b].
\]
The determiners have type
\[
    \langle\langle e,t\rangle,\langle\langle e,t\rangle,t\rangle\rangle
\]
and, writing their restrictor and scope as subsets $P,Q\subseteq\mathcal{D}_e$, are interpreted by
\[
    \textit{every}(P)(Q)=[P\subseteq Q],\qquad
    \textit{some}(P)(Q)=[P\cap Q\neq\emptyset],\qquad
    \textit{no}(P)(Q)=[P\cap Q=\emptyset],
\]
\[
    \textit{exactly one}(P)(Q)=[\,|P\cap Q|=1\,],\qquad
    \textit{at least two}(P)(Q)=[\,|P\cap Q|\ge2\,],
\]
and, on finite entity domains,
\[
    \textit{most}(P)(Q)
    =
    [\,|P\cap Q|>|P\setminus Q|\,].
\]
Names have type $e$ and predicates type $\langle e,t\rangle$, with arbitrary intensions.

\begin{proposition}[Entity sort]\label{prop:entity}
    Let the entity domains be nonempty, and let each item below be standardly interpreted in both models, with constant intension where the item is a determiner or identity. Then (M2) holds for identity iff $f_e$ is injective; for \textit{every}, \textit{some}, and \textit{no} iff $f_e$ is surjective; for \textit{exactly one} and \textit{at least two} iff $f_e$ is bijective; and, on finite entity domains, for \textit{most} iff $f_e$ is surjective and all fibers have the same size.
    For a name $c$, (M2) holds iff
    \[
        f_e(\mathcal{I}(c)(s))
        =
        \mathcal{I}'(c)(f_s(s))
    \]
    for every $s$; for a predicate $P$, iff
    \[
        \mathcal{I}(P)(s)
        =
        \mathcal{I}'(P)(f_s(s))\circ f_e
    \]
    for every $s$. Standard truth-functional connectives satisfy (M2) under every sort map.
\end{proposition}

\begin{proof}
    At type $\langle e,t\rangle$,
    \[
        P\rel P'
        \quad\Longleftrightarrow\quad
        P=f^{-1}(P').
    \]
    Hence determiners $\Delta,\Delta'$ are related iff
    \[
        \Delta(f^{-1}(P'))(f^{-1}(Q'))
        =
        \Delta'(P')(Q')
    \]
    for all $P',Q'\subseteq\mathcal{D}'_e$.

    For identity, relatedness is
    \[
        [a=b]=[f_e(a)=f_e(b)]
    \]
    for all $a,b$, which is equivalent to injectivity.

    For \textit{every}, relatedness requires
    \[
        [f^{-1}(P')\subseteq f^{-1}(Q')]
        =
        [P'\subseteq Q']
    \]
    for all $P',Q'$. Surjectivity makes preimage reflect inclusion. Conversely, if $b\notin\IM f_e$, taking $P'=\{b\}$ and $Q'=\emptyset$ separates the two sides. For \textit{some},
    \[
        [f^{-1}(P'\cap Q')\neq\emptyset]
        =
        [P'\cap Q'\neq\emptyset],
    \]
    which holds for all $P',Q'$ exactly when $f_e$ is surjective; \textit{no} is its negation.

    For \textit{exactly one}, the condition reduces to
    \[
        [\,|f^{-1}(A)|=1\,]
        =
        [\,|A|=1\,]
    \]
    for every $A\subseteq\mathcal{D}'_e$. Applying this to singletons shows that every fiber has exactly one element, hence $f_e$ is bijective; the converse is immediate.

    For \textit{at least two},
    \[
        [\,|f^{-1}(A)|\ge2\,]
        =
        [\,|A|\ge2\,].
    \]
    Singletons show that every fiber has size at most one. If some fiber over $b$ were empty, nonemptiness of $\mathcal{D}_e$ would give a point $b'\neq b$ with a singleton fiber, and $A=\{b,b'\}$ would violate the equation. Thus $f_e$ is bijective, and the converse is immediate.

    For \textit{most}, suppose first that $f_e$ is surjective and every fiber has size $k$. Then
    \[
        |f^{-1}(A)|=k|A|
    \]
    for every $A$, so the majority comparison is preserved. Conversely, if the fiber over $b$ is empty and the fiber over $b'$ is nonempty, take
    \[
        P'=\{b,b'\},\qquad Q'=\{b'\}.
    \]
    The majority condition fails in $\mathcal{D}'_e$ and holds in $\mathcal{D}_e$. If instead two nonempty fibers have sizes $k_1>k_2$, take
    \[
        P'=\{b_1,b_2\},\qquad Q'=\{b_1\},
    \]
    and again the two majority judgments differ. Thus $f_e$ is surjective with uniform fiber size.

    The claims for names and predicates are the covariant and contravariant forms of (M2), respectively. Standard truth functions agree in the two models.
\end{proof}

These conditions define distinct classes of entity maps: uniform-fiber surjections lie strictly between arbitrary surjections and bijections, and injectivity is independent of surjectivity. The characterization of \textit{most} is restricted to finite domains. Measure- or cardinal-based treatments on infinite domains need not have the same invariance class.

\begin{proposition}[Index sorts]\label{prop:index}
    Let $\sigma\in\Sigma$, and let each item below be standardly interpreted in both models, with constant intension. Then:
    \begin{enumerate}[label=(\arabic*)]
        \item (M2) holds for $\B_\sigma$ iff $f_\sigma$ is a bounded morphism;
        \item for identity $=_\sigma$ iff $f_\sigma$ is injective;
        \item for
        \[
            \forall_\sigma(P)=[P=\mathcal{D}_\sigma]
        \]
        iff $f_\sigma$ is surjective;
        \item for
        \[
            r_\sigma(w)(v)=[\mathcal{R}_\sigma(w,v)]
        \]
        iff $f_\sigma$ reflects accessibility:
        \[
            \mathcal{R}_\sigma(w,v)
            \quad\Longleftrightarrow\quad
            \mathcal{R}'_\sigma(f_\sigma(w),f_\sigma(v)).
        \]
    \end{enumerate}
    If $f_\sigma$ is already a bounded morphism, the counting operator
    \[
        \B^{\ge2}_\sigma(s)(p)
        =
        \left[
        \left|
            \{w'\in\mathcal{R}_\sigma(s_\sigma):
            p(s[\sigma\mapsto w'])=1\}
        \right|
        \ge2
        \right]
    \]
    is related iff $f_\sigma$ is injective on every accessible set $\mathcal{R}_\sigma(w)$.

    On finite frames, define
    \[
        \M_\sigma(s)(p)
        =
        [\,|\mathcal{R}_\sigma(s_\sigma)\cap p_s|
        >
        |\mathcal{R}_\sigma(s_\sigma)\setminus p_s|\,],
    \]
    where
    \[
        p_s:=\{w':p(s[\sigma\mapsto w'])=1\}.
    \]
    Given that $f_\sigma$ is bounded, $\M_\sigma$ is related iff, for every $w$, writing
    \[
        n:=|\mathcal{R}'_\sigma(f_\sigma(w))|,
        \qquad
        N:=|\mathcal{R}_\sigma(w)|,
    \]
    the $\lfloor n/2\rfloor$ largest fibers over
    $\mathcal{R}'_\sigma(f_\sigma(w))$ contain at most $N/2$ points in total, and strictly fewer than $N/2$ when $n$ is odd.
\end{proposition}

\begin{proof}
    The claim for necessity is Lemma~\ref{lem:box}. For identity, relatedness is
    \[
        [w=v]
        =
        [f_\sigma(w)=f_\sigma(v)],
    \]
    hence injectivity. For $r_\sigma$, it is exactly
    \[
        \mathcal{R}_\sigma(w,v)
        =
        \mathcal{R}'_\sigma(f_\sigma(w),f_\sigma(v)).
    \]
    For $\forall_\sigma$, the condition is
    \[
        [f^{-1}(P')=\mathcal{D}_\sigma]
        =
        [P'=\mathcal{D}'_\sigma]
    \]
    for every $P'\subseteq\mathcal{D}'_\sigma$, which is equivalent to surjectivity as in Proposition~\ref{prop:entity}.

    For the two modal operators, fix $s$ and write $w=s_\sigma$. As in Lemma~\ref{lem:box}, relatedness reduces to comparing the operator in $\mathcal{M}$ on $p'\circ f_s$ with the operator in $\mathcal{M}'$ on $p'$. As $p'$ varies,
    \[
        A
        :=
        \{v':p'(f_s(s)[\sigma\mapsto v'])=1\}
    \]
    ranges over all subsets of $\mathcal{D}'_\sigma$.

    For $\B^{\ge2}_\sigma$, relatedness is therefore
    \[
        [\,|\mathcal{R}_\sigma(w)\cap f^{-1}(A)|\ge2\,]
        =
        [\,|\mathcal{R}'_\sigma(f_\sigma(w))\cap A|\ge2\,].
    \]
    Since boundedness gives a surjection
    \[
        f_\sigma:
        \mathcal{R}_\sigma(w)
        \twoheadrightarrow
        \mathcal{R}'_\sigma(f_\sigma(w)),
    \]
    Proposition~\ref{prop:entity} for \textit{at least two} shows that the equation holds iff this restriction is injective.

    For $\M_\sigma$, put
    \[
        R:=\mathcal{R}_\sigma(w),
        \qquad
        R':=\mathcal{R}'_\sigma(f_\sigma(w)),
    \]
    and, for $b\in R'$,
    \[
        k_b:=|R\cap f^{-1}(b)|,
        \qquad
        K(A):=\sum_{b\in A}k_b.
    \]
    Boundedness gives $k_b\ge1$ and
    \[
        \sum_{b\in R'}k_b=N.
    \]
    Relatedness becomes
    \[
        [\,K(A)>N/2\,]
        =
        [\,|A|>n/2\,]
    \]
    for every $A\subseteq R'$.

    Let $L$ be the sum of the $\lfloor n/2\rfloor$ largest fibers. If $n$ is even, the equation for a set of the $n/2$ largest fibers gives $L\le N/2$. If $n$ is odd, its complement has more than half the points of $R'$, so $L<N/2$.

    Conversely, suppose these bounds on $L$ hold. If $n$ is even, every $A$ with $|A|\le n/2$ satisfies $K(A)\le L\le N/2$, while if $|A|>n/2$, its complement has fewer than $n/2$ points and therefore weight less than $N/2$. The odd case is analogous, using $L<N/2$. Hence the majority judgments agree for every $A$.
\end{proof}

Table~\ref{tab:determination} summarizes the two propositions. The closest parallel is local: boundedness makes
\[
    f_\sigma:
    \mathcal{R}_\sigma(w)
    \twoheadrightarrow
    \mathcal{R}'_\sigma(f_\sigma(w))
\]
a surjection, and counting modalities strengthen this to a bijection on each accessible set. Counting determiners force bijectivity on the entity domain in the same way. Proportionality breaks the parallel: for \textit{most}, the restrictor ranges over arbitrary subsets of the entity domain, which forces uniform fibers. For $\M_\sigma$, the frame fixes the restrictor as an accessible set. Only majority judgments within that set must then be preserved. Uniform fibers suffice. On accessible sets of odd size, the bound of Proposition~\ref{prop:index} also admits nonuniform fibers, such as sizes 1, 2, and 2 over three points.

\begin{table}[h]
    \centering
    \caption{Conditions equivalent to relatedness of standardly interpreted items}
    \label{tab:determination}
    \small
    \begin{tabular}{@{}lp{3.1cm}lp{5.75cm}@{}}
        \toprule
        \textbf{Entity item} & \textbf{Condition on $f_e$} & \textbf{Index item} & \textbf{Condition on $f_\sigma$} \\
        \midrule
        predicate, name & pullback, transport & propositional constant & pullback \\
        \textit{every}, \textit{some}, \textit{no} & surjective & $\B_\sigma$ & bounded \\
        \textit{most} & surjective, uniform fibers & $\M_\sigma$ (given bounded) & preserves and reflects accessible-set majorities \\
        \textit{at least two}, \textit{exactly one} & bijective & $\B^{\ge2}_\sigma$ (given bounded) & injective on each $\mathcal{R}_\sigma(w)$ \\
        identity & injective & $=_\sigma$ & injective \\
         & & $\forall_\sigma$ & surjective \\
         & & $r_\sigma$ & reflects accessibility \\
        \bottomrule
    \end{tabular}
\end{table}

\begin{corollary}[Rigidity]\label{cor:rigidity}
    Let $\Lex$ contain, standardly interpreted, either \textit{exactly one} or \textit{at least two}, or identity together with one of \textit{every}, \textit{some}, or \textit{no}. Then every morphism between models over $\Sigma$ and $\Lex$ is bijective on the entity sort. If $\Sigma=\emptyset$, every morphism is an isomorphism, so $\Mod_\emptyset(\Lex)$ is a groupoid.
\end{corollary}

\begin{proof}
    Bijectivity of $f_e$ follows from Proposition~\ref{prop:entity}. When $\Sigma=\emptyset$, a sort map consists only of $f_e$. If $f_e$ is bijective, an induction on types shows that the induced relation at every type is the graph of a bijection, and that the relation induced by $f_e^{-1}$ is its converse. Hence (M2) for $f_e$ implies (M2) for $f_e^{-1}$, so every morphism is an isomorphism.
\end{proof}

For any lexicon satisfying the hypotheses of Corollary~\ref{cor:rigidity}, the entity component of every morphism is bijective. For such lexicons, we henceforth choose representatives with a common entity domain within each isomorphism class and restrict to morphisms with $f_e=\mathrm{id}$. Under this convention, the induced relation is equality at every extensional type.

\begin{remark}[The definite article]\label{rem:article}
    Heim and Kratzer interpret \textit{the} as a partial function defined on singleton predicates \cite{Heim1998}. Definition~\ref{def:model} uses total functions. If undefinedness is represented by a designated $\#\in\mathcal{D}_e$, with
    \[
        \iota(P)
        =
        \begin{cases}
            \text{the unique member of }P, & |P|=1,\\
            \#, & \text{otherwise},
        \end{cases}
    \]
    then relatedness requires $f_e(\#)=\#'$ and makes $f_e$ a bijection away from the preimage of $\#'$. Thus the totalized article rigidifies the ordinary entity values much as \textit{exactly one} does. The treatment of the designated undefined value depends on the chosen totalization.
\end{remark}

\begin{remark}[Logicality]\label{rem:logicality}
    Proposition~\ref{prop:entity} separates familiar invariance criteria: invariance under permutations \cite{Tarski1986} and under bijections between domains \cite{Sher1991} holds for all of the standard items above; Feferman's homomorphism criterion excludes identity and the counting determiners and retains \textit{every}, \textit{some}, and \textit{no} \cite{Feferman1999}. On finite domains, \textit{most} instead selects uniform-fiber surjections. The proposition identifies the class of maps selected by each lexical fragment. Which criterion characterizes logicality is a separate question. On index sorts, the same analysis selects bounded morphisms for necessity.
\end{remark}

\section{Extensional models}\label{sec:extensional}

Under the representative convention of Section~\ref{sec:determination}, which covers the lexicons of Corollary~\ref{cor:rigidity}, morphisms between zero-index models have $f_e=\mathrm{id}$. Before this choice of representatives, they form the groupoid of that corollary. For an extensional lexicon, these zero-index models are the usual extensional models. Intensional models on singleton frames recover the same structure once the relation on each singleton is fixed.

\begin{proposition}[Singleton frames]\label{prop:singleton}
    Let $\Sigma$ be nonempty and let $\Lex$ have extensional types. The full subcategory of $\Mod_\Sigma(\Lex)$ on models whose index domains are singletons is the disjoint union of $2^{|\Sigma|}$ full subcategories, indexed by the choice of an empty or reflexive relation on each singleton. There are no morphisms between distinct components, and each component is isomorphic to $\Mod_\emptyset(\Lex)$.
\end{proposition}

\begin{proof}
    Between two singleton frames, there is a unique function, bounded iff their relations agree: forth fails from a reflexive point to an irreflexive one; back fails in the reverse direction. Once these relations agree, (M2) identifies the extensional lexical interpretations with those of the corresponding extensional models. Thus, fixing the singleton relations leaves exactly the data of an extensional model.
\end{proof}

The identification with $\Sigma=\emptyset$ avoids this choice of singleton relations. For an extensional lexicon, evaluating intensions at a designated index gives the corresponding extensional model.

\begin{definition}[Pointed models]\label{def:pointed}
    A \textbf{pointed model} is a pair $(\mathcal{M},s)$ with $s\in S$. A morphism
    \[
        (\mathcal{M},s)\to(\mathcal{M}',s')
    \]
    is a morphism $f:\mathcal{M}\to\mathcal{M}'$ satisfying $f_s(s)=s'$. These form the category $\Mod^\bullet_\Sigma(\Lex)$.

\end{definition}

    Suppose $\Lex$ has extensional types. Define
    \[
        \mathrm{ev}(\mathcal{M},s)
        :=
        \langle
            (\mathcal{D}_\tau)_{\tau\in\Text},
            \mathcal{I}_s
        \rangle,
        \qquad
        \mathcal{I}_s(c):=\mathcal{I}(c)(s).
    \]
    For a choice $\rho=(\rho_\sigma)_{\sigma\in\Sigma}$ of an empty or reflexive relation on each singleton, let $\iota_\rho$ send an extensional model to the corresponding singleton-frame model with constant intensions, and let $\iota^\bullet_\rho$ point it at its unique index.

\begin{proposition}[Evaluation]\label{prop:evaluation}
    Let $\Lex$ have extensional types. If
    \[
        f:(\mathcal{M},s)\to(\mathcal{M}',s')
    \]
    is a morphism of pointed models, then $f_e$ is a morphism
    \[
        \mathrm{ev}(\mathcal{M},s)
        \to
        \mathrm{ev}(\mathcal{M}',s'),
    \]
    and the two models are equal when $f_e=\mathrm{id}$. Hence $\mathrm{ev}:\Mod^\bullet_\Sigma(\Lex)\to\Mod_\emptyset(\Lex)$, with $\mathrm{ev}(f):=f_e$, is a functor. Moreover,
    \[
        \mathrm{ev}\circ\iota^\bullet_\rho
        =
        \mathrm{id}_{\Mod_\emptyset(\Lex)},
    \]
    and the images of the $\iota_\rho$ are precisely the components of Proposition~\ref{prop:singleton}.
\end{proposition}

\begin{proof}
    For $c\in\Lex_\tau$ with $\tau$ extensional, the relation $\rel^{f}_\tau$ depends only on $f_e$, so (M2) and $f_s(s)=s'$ give
    \[
        \mathcal{I}(c)(s)
        \rel^{f_e}_\tau
        \mathcal{I}'(c)(s'),
    \]
    which is (M2) for $f_e$ between the evaluated models. When $f_e=\mathrm{id}$, this relation is equality. The remaining claims follow directly from the definitions.
\end{proof}

\begin{figure}[h]
    \centering
    \begin{tikzcd}[column sep=6em, row sep=4em]
        \Mod^{\bullet}_\Sigma(\Lex) \arrow[r, shift left=0.7ex, "\mathrm{ev}"] \arrow[d, "\text{forget}"'] &
        \Mod_\emptyset(\Lex) \arrow[l, shift left=0.7ex, "\iota^{\bullet}_\rho"] \arrow[dl, "\iota_\rho"] \\
        \Mod_\Sigma(\Lex) &
    \end{tikzcd}
    \caption{Evaluation and the singleton-frame embedding.}
    \label{fig:categories}
\end{figure}

Now, when does an arbitrary intensional model admit a morphism to one in which some index sorts have been collapsed?

\begin{definition}[Collapse]\label{def:collapse}
    Let $\mathcal{M}$ be a model over $\Sigma$ and $\Lex$, and let $\Sigma_0\subseteq\Sigma$. A \textbf{$\Sigma_0$-collapse} of $\mathcal{M}$ is a model $\mathcal{M}'$ over the same $\Sigma$ and $\Lex$ such that $\mathcal{D}'_e=\mathcal{D}_e$, each $\mathcal{D}'_\sigma$ is a singleton for $\sigma\in\Sigma_0$, and $\mathcal{F}'_\sigma=\mathcal{F}_\sigma$ for $\sigma\notin\Sigma_0$, together with the sort map $f$ that is the identity on $e$ and on the sorts outside $\Sigma_0$ and the unique map on the sorts in $\Sigma_0$.
\end{definition}

Two indices $s,u\in S$ are $\Sigma_0$-variants if they agree outside $\Sigma_0$; these are exactly the fibers of $f_s$.

\begin{proposition}[Inert sorts]\label{prop:inert}
    Let every type in $\Lex$ be extensional or covariant, let $\mathcal{M}$ be a model over $\Sigma$ and $\Lex$, and let $\Sigma_0\subseteq\Sigma$. The collapse map $f$ is a morphism from $\mathcal{M}$ to some $\Sigma_0$-collapse $\mathcal{M}'$ iff:
    \begin{enumerate}[label=(\arabic*)]
        \item for every $\sigma\in\Sigma_0$, $\mathcal{R}_\sigma$ is serial or empty;
        \item for every $c\in\Lex_\tau$, the map
        \[
            f^+_\tau\circ\mathcal{I}(c)
        \]
        is constant on $\Sigma_0$-variants, with $f^+_\tau=\mathrm{id}$ when $\tau$ is extensional.
    \end{enumerate}
    In that case, $\mathcal{M}'$ is determined by $\mathcal{M}$: the collapsed point for $\sigma$ is reflexive when $\mathcal{R}_\sigma$ is serial and irreflexive when it is empty, and
    \[
        \mathcal{I}'(c)(f_s(s))
        =
        f^+_\tau(\mathcal{I}(c)(s)).
    \]
\end{proposition}

\begin{proof}
    For $\sigma\in\Sigma_0$, let $\rho_\sigma$ be the relation on the target singleton. The unique map to that point satisfies forth iff $\mathcal{R}_\sigma$ is empty or $\rho_\sigma$ is reflexive, and back iff $\rho_\sigma$ is empty or $\mathcal{R}_\sigma$ is serial. Since the source frame is nonempty, (M1) therefore holds exactly when
    \[
        \mathcal{R}_\sigma\text{ is serial and }\rho_\sigma\text{ is reflexive},
    \]
    or
    \[
        \mathcal{R}_\sigma\text{ is empty and }\rho_\sigma\text{ is empty}.
    \]

    For $c\in\Lex_\tau$ of covariant type, (M2) is
    \[
        f^+_\tau(\mathcal{I}(c)(s))
        =
        \mathcal{I}'(c)(f_s(s)).
    \]
    Hence, the left side must agree on every fiber of $f_s$, that is, on $\Sigma_0$-variants.

    Conversely, suppose the two conditions do hold. Choose the relation on each collapsed singleton as above and define
    \[
        \mathcal{I}'(c)(s')
        :=
        f^+_\tau(\mathcal{I}(c)(s))
    \]
    for any $s$ with $f_s(s)=s'$. Such an $s$ exists because $f_s$ is surjective, and the value is independent of the choice by condition~(2). Thus both (M1) and (M2) hold.
\end{proof}

For extensional $\tau$, condition~(2) says simply that $\mathcal{I}(c)$ is independent of the coordinates in $\Sigma_0$, as shown in Figure~\ref{fig:collapse}. For covariant higher types, the same condition applies after transport by $f^+_\tau$. Contravariant lexical types require an additional image condition under $f^-_\tau$. Proposition~\ref{prop:inert} excludes them for that reason. In particular, if the accessibility constant $r_\sigma$ is present and $\sigma\in\Sigma_0$, collapse requires $\mathcal{R}_\sigma$ to be empty or total.

\begin{figure}[h]
    \centering
    \begin{tikzcd}[column sep=large, row sep=large]
        S \arrow[rr, "f_s"] \arrow[dr, "\mathcal{I}(c)"'] & & S' \arrow[dl, "\mathcal{I}'(c)"] \\
        & \mathcal{D}_\tau &
    \end{tikzcd}
    \caption{For extensional $c$, collapse is possible exactly when $\mathcal{I}(c)$ factors through $f_s$.}
    \label{fig:collapse}
\end{figure}

\begin{corollary}[Extensional as singletons]\label{cor:shadow}
    Let $\Lex$ have extensional or covariant types, and let $\mathcal{M}$ be a model over $\Sigma$ and $\Lex$. Then $\mathcal{M}$ admits a morphism with identity entity component to a model whose index domains are all singletons iff every $\mathcal{R}_\sigma$ is serial or empty and, for every $c\in\Lex_\tau$, the map $f^+_\tau\circ\mathcal{I}(c)$ is constant on $S$, where $f$ is the $\Sigma$-collapse.
    The target has
    \[
        \mathcal{I}_{ext}(c)
        :=
        f^+_\tau(\mathcal{I}(c)(s)),
    \]
    independent of $s$. For an extensional lexicon this target is
    \[
        \iota_\rho(\mathrm{ev}(\mathcal{M},s))
    \]
    for every $s$, with $\rho$ determined by the source frames as in Proposition~\ref{prop:inert}.

    If $\mathcal{R}_\sigma$ is serial and the conditions of Proposition~\ref{prop:inert} hold for $\Sigma_0=\{\sigma\}$, then
    \[
        \nec[\sigma]\varphi\leftrightarrow\varphi
    \]
    is valid in $\mathcal{M}$ for every closed formula $\varphi$.
\end{corollary}

\begin{proof}
    A morphism with identity entity component into a model with singleton index domains is the $\Sigma$-collapse map, so the first claim is Proposition~\ref{prop:inert} with $\Sigma_0=\Sigma$. For the final claim, let
    \[
        f:\mathcal{M}\to\mathcal{M}'
    \]
    be the $\{\sigma\}$-collapse. Since $\mathcal{R}_\sigma$ is serial, the collapsed point is reflexive. Hence in $\mathcal{M}'$,
    \[
        \den{\nec[\sigma]\varphi}^{\mathcal{M}',g',s'}
        =
        \den{\varphi}^{\mathcal{M}',g',s'},
    \]
    so $\nec[\sigma]\varphi\leftrightarrow\varphi$ is valid there. Corollary~\ref{cor:transfer} transfers validity back along $f$.
\end{proof}

\begin{remark}[Terminal objects]\label{rem:terminal}
    Under the coalgebraic presentation of frames, $\Krb$ is the category of coalgebras for the covariant powerset functor \cite{Rutten2000}; it has no terminal object: a final powerset coalgebra would have a bijective structure map \cite{Lambek1968}, contradicting Cantor's theorem. The full subcategory of serial frames has the reflexive point as a terminal object, and the full subcategory of frames with empty accessibility has the irreflexive point. These are precisely the two frame-theoretic cases in Proposition~\ref{prop:inert}.
\end{remark}

\section{Change of lexicon}\label{sec:change}

Montague's intensional logic accesses indices through modal operators; Gallin's two-sorted type theory also permits quantification over indices and an accessibility predicate \cite{Gallin2016}. Definition~\ref{def:denotation} already admits index variables, so adding the corresponding quantifier and accessibility constants changes the morphism class.

\begin{proposition}[Two-sorted lexicon]\label{prop:twosorted}
    Let $\Lex$ have polarized types, and define
    \[
        \Lex^\dagger
        :=
        \Lex\cup\{\forall_\sigma,r_\sigma:\sigma\in\Sigma\},
    \]
    with the new constants standardly interpreted as in Proposition~\ref{prop:index}. Let $\Mod_\Sigma(\Lex^\dagger)$ be the full subcategory of models with these interpretations. Every model over $\Lex$ has a unique such expansion, and under this correspondence $\Mod_\Sigma(\Lex^\dagger)$ is the wide subcategory of $\Mod_\Sigma(\Lex)$ whose morphisms are those $f$ for which every $f_\sigma$ is surjective and reflects accessibility. Each such $f_\sigma$ is a bounded morphism. If $\Sigma\neq\emptyset$, the inclusion is proper.
\end{proposition}

\begin{proof}
    The new constants have polarized types, so Theorem~\ref{thm:category} applies, and their interpretations are determined by the frames, giving the claimed correspondence on objects.

    By Proposition~\ref{prop:index}, (M2) for $\forall_\sigma$ is equivalent to surjectivity of $f_\sigma$, while (M2) for $r_\sigma$ is equivalent to reflection of accessibility. Such a map is bounded: forth is one direction of accessibility reflection, and for back, if
    \[
        \mathcal{R}'_\sigma(f_\sigma(w),v'),
    \]
    surjectivity gives $v$ with $f_\sigma(v)=v'$, and reflection gives $\mathcal{R}_\sigma(w,v)$.

    For properness, fix $\sigma\in\Sigma$. Let $\mathcal{M}'$ have the reflexive singleton frame at $\sigma$, and let $\mathcal{M}$ instead have the two-point cycle
    \[
        \{w,v\},
        \qquad
        \mathcal{R}_\sigma=\{(w,v),(v,w)\},
    \]
    with the remaining frames unchanged. Let $f$ collapse this cycle to the singleton and be the identity elsewhere. By Lemma~\ref{lem:surjective}, choose the required preimages for covariant lexical types; define the interpretations in $\mathcal{M}$ so that (M2) holds, using pullback for contravariant types and such preimages for covariant types. Since the two-point cycle is serial, its collapse to the reflexive point is bounded. Thus, $f$ is a morphism in $\Mod_\Sigma(\Lex)$, but it does not reflect accessibility, since
    \[
        \mathcal{R}'_\sigma(f_\sigma(w),f_\sigma(w))
    \]
    holds while $\mathcal{R}_\sigma(w,w)$ does not. Hence $f$ is not a morphism in $\Mod_\Sigma(\Lex^\dagger)$.
\end{proof}

The separating formula
\[
    \forall_\sigma(\lambda w.r_\sigma(w)(w))
\]
states that every index of sort $\sigma$ is reflexive; it holds in the singleton frame, and fails in the two-point cycle above.

Proposition~\ref{prop:twosorted} is a model-theoretic counterpart of the additional expressive resources of Gallin's two-sorted system. Zimmermann shows that the part of $\mathrm{Ty}_2$ not retranslatable into intensional logic is confined to a decidable class without linguistic application \cite{Zimmermann1989}. That definability result concerns translation between languages. The proposition compares morphism classes over fixed models. The two-sorted definition of necessity from $\forall_\sigma$ and $r_\sigma$ requires a term denoting the current index of evaluation, and the language of Definition~\ref{def:denotation} has none.

The resulting frame maps are surjective strong homomorphisms: they are surjective, and preserve and reflect the accessibility relation \cite{Chang1990}.

\section{Several sorts}\label{sec:sorts}

A model over $\Sigma$ carries one frame for each sort; its modal operators act on the product of the frame domains, and this product construction is functorial.

\begin{definition}[Relational frames]\label{def:relframes}
    A \textbf{$\Sigma$-frame} is a nonempty set $X$ equipped with binary relations $(R_\sigma)_{\sigma\in\Sigma}$. A \textbf{$\Sigma$-bounded morphism} is a function that is a bounded morphism with respect to each $R_\sigma$; these form the category $\Fr_\Sigma$.

    For a family of frames $(\mathcal{D}_\sigma,\mathcal{R}_\sigma)_{\sigma\in\Sigma}$, define
    \[
        \Pi(\mathcal{F})
        :=
        (S,(\mathcal{R}^S_\sigma)_\sigma),
        \qquad
        S:=\prod_{\sigma\in\Sigma}\mathcal{D}_\sigma,
    \]
    where
    \[
        \mathcal{R}^S_\sigma(s,u)
    \]
    iff
    \[
        u=s[\sigma\mapsto v]
    \]
    for some $v$ with $\mathcal{R}_\sigma(s_\sigma,v)$. For a family of maps $f=(f_\sigma)_\sigma$, let
    \[
        \Pi(f):=\prod_\sigma f_\sigma.
    \]
\end{definition}

\begin{figure}[h]
    \centering
    \begin{tikzcd}[column sep=large]
        \Mod_\Sigma(\Lex) \arrow[r, "U"] &
        \prod_{\sigma\in\Sigma}\Krb
        \arrow[r, "\Pi"] &
        \Fr_\Sigma
    \end{tikzcd}
    \qquad\qquad
    \begin{tikzcd}[column sep=large, row sep=large]
        S \arrow[r, "\Pi(f)"] \arrow[d, "\mathcal{R}^{S}_\sigma"', dashed] &
        S' \arrow[d, "\mathcal{R}'^{S}_\sigma", dashed] \\
        S \arrow[r, "\Pi(f)"'] &
        S'
    \end{tikzcd}
    \caption{The forgetful and product functors; on the right, boundedness of $\Pi(f)$ at sort $\sigma$.}
    \label{fig:product}
\end{figure}

\begin{proposition}[Product functor]\label{prop:product}
    The product construction defines a faithful functor
    \[
        \Pi:\prod_{\sigma\in\Sigma}\Krb\to\Fr_\Sigma.
    \]
    In a model $\mathcal{M}$, $\B_\sigma$ is the box operator induced by $\mathcal{R}^S_\sigma$, and the frame components of a model and morphism define a forgetful functor
    \[
        U:\Mod_\Sigma(\Lex)\to\prod_{\sigma\in\Sigma}\Krb.
    \]
    For distinct $\sigma,\sigma'\in\Sigma$, every model validates
    \[
        \pos[\sigma]\pos[\sigma']\varphi
        \leftrightarrow
        \pos[\sigma']\pos[\sigma]\varphi
    \]
    and
    \[
        \pos[\sigma]\nec[\sigma']\varphi
        \to
        \nec[\sigma']\pos[\sigma]\varphi.
    \]
\end{proposition}

\begin{proof}
    Let $f=(f_\sigma)_\sigma$ be a family of bounded morphisms. Suppose
    \[
        \mathcal{R}^S_\sigma
        \bigl(s,s[\sigma\mapsto v]\bigr),
    \]
    so $\mathcal{R}_\sigma(s_\sigma,v)$. Then
    \[
        \Pi(f)(s[\sigma\mapsto v])
        =
        \Pi(f)(s)[\sigma\mapsto f_\sigma(v)],
    \]
    and forth for $f_\sigma$ gives
    \[
        \mathcal{R}'_\sigma
        \bigl(f_\sigma(s_\sigma),f_\sigma(v)\bigr).
    \]
    Thus $\Pi(f)$ satisfies forth for $\mathcal{R}^S_\sigma$.

    Conversely, if
    \[
        \mathcal{R}'^S_\sigma(\Pi(f)(s),u'),
    \]
    then
    \[
        u'=\Pi(f)(s)[\sigma\mapsto v']
    \]
    for some $v'$ with
    \[
        \mathcal{R}'_\sigma(f_\sigma(s_\sigma),v').
    \]
    Back for $f_\sigma$ yields $v$ with
    \[
        \mathcal{R}_\sigma(s_\sigma,v),
        \qquad
        f_\sigma(v)=v'.
    \]
    Setting $u:=s[\sigma\mapsto v]$ gives
    \[
        \mathcal{R}^S_\sigma(s,u)
        \quad\text{and}\quad
        \Pi(f)(u)=u'.
    \]
    Hence, $\Pi(f)$ is a $\Sigma$-bounded morphism.

    Identities and composition are preserved componentwise. Faithfulness follows because equality of the product maps implies equality of each component.

    By Definition~\ref{def:denotation}(5),
    \[
        \B_\sigma(s)(p)=1
    \]
    iff $p(u)=1$ for every $u$ with $\mathcal{R}^S_\sigma(s,u)$, so $\B_\sigma$ is the box operator of $\mathcal{R}^S_\sigma$. The functor $U$ is given by condition (M1).

    For distinct $\sigma$ and $\sigma'$, substitutions in the two coordinates commute:
    \[
        s[\sigma\mapsto v][\sigma'\mapsto v']
        =
        s[\sigma'\mapsto v'][\sigma\mapsto v].
    \]
    This gives the commutation. For confluence, if
    \[
        \pos[\sigma]\nec[\sigma']\varphi
    \]
    holds at $s$, choose $v\in\mathcal{R}_\sigma(s_\sigma)$ such that $\nec[\sigma']\varphi$ holds at $s[\sigma\mapsto v]$. The same $v$ works after any $\sigma'$-move, yielding
    \[
        \nec[\sigma']\pos[\sigma]\varphi.
    \]
\end{proof}

These interaction principles are the characteristic commutation and confluence conditions of product modal frames \cite{gabbay2003many}. The construction considered here is narrower than an arbitrary $\Sigma$-frame: each $\mathcal{R}^S_\sigma$ changes only the $\sigma$-coordinate. Thus, for example, a relation on $\mathcal{D}_\iota\times\mathcal{D}_\ell$ defined by
\[
    (t,x)\,R\,(t',x')
    \quad\text{iff}\quad
    |x'-x|\le t'-t
\]
is not in the image of $\Pi$, since a single accessibility step may change both coordinates. Such frames lie outside the present model class.

\section{Conclusion}\label{sec:conclusion}

A morphism of intensional models is a sort map under which the lexical interpretation is invariant. The lexicon determines the admissible maps: counting determiners force bijections on entities, necessity selects bounded morphisms on index frames, index quantification together with accessibility strengthens these to surjective strong homomorphisms, and proportional modalities impose preservation and reflection of majorities within accessible sets. The construction yields a category when the lexical types are composable. Ordinary extensional models arise when $\Sigma=\emptyset$ and the lexicon has extensional types; inert sorts can collapse along morphisms under the conditions of Proposition~\ref{prop:inert}; and several index sorts combine through the product functor of Proposition~\ref{prop:product}.

Proposition~\ref{prop:index} assigns an invariance class to a typed item. The converse problem we leave for future work: to characterize which classes of frame maps arise from finite typed lexicons; likewise, whether composability of lexical types, which suffices for closure of morphisms under composition, is also necessary.

The finite proportional case also suggests an extension to the measure-theoretic frames of \cite{quigley2026intensional}. Any such extension would require an appropriate notion of morphism compatible with measure-sensitive operators. Proposition~\ref{prop:index} indicates that the relevant condition need not reduce to uniform fibers. Partial interpretations likewise require a treatment beyond the total models used here. These extensions, together with the vector logic of \cite{Quigley2025,quigley2026intensional}, are natural next directions for future work.

\bibliography{vector_logic_category}

\end{document}